\documentclass[11pt]{amsart}
\usepackage{amsmath,amsfonts,amsthm,mathrsfs,amssymb,amscd,comment,enumerate,amsxtra,url,tikz-cd}
\input{xypic}
\xyoption{all}

\usepackage{xcolor} 
\colorlet{mdtRed}{red!50!black}
\definecolor{dblue}{rgb}{0,0,.6}
\usepackage[colorlinks]{hyperref}
\hypersetup{linkcolor=blue,citecolor=dblue,filecolor=dullmagenta,urlcolor=mdtRed}

\numberwithin{equation}{section}
\newtheorem{theorem}[equation]{Theorem}

\newtheorem{lemma}[equation]{Lemma}

\newtheorem{definition}[equation]{Definition}

\newtheorem*{theorem*}{Theorem}
\newtheorem*{corollary*}{Corollary}
\newtheorem*{proposition*}{Proposition}

\theoremstyle{remark}

\def\subsection{
	\refstepcounter{equation}
	\noindent {\bf \arabic{section}.\arabic{equation}.}
}

\newcommand{\Z}{\mathbb{Z}}
\newcommand{\C}{\mathbb{C}}

\renewcommand{\P}{\mathbb{P}}

\newcommand{\mb}[1]{\mathbb{#1}}
\newcommand{\mc}[1]{\mathcal{#1}}

\usepackage{marginnote}
\usetikzlibrary{calc}

\begin{document}
	
\title[Ulrich bundles on general double plane covers]{Rank 2 Ulrich bundles on general double plane covers} 

\author[R. Sebastian]{Ronnie Sebastian} 

\address{Department of Mathematics, Indian Institute of Technology Bombay, Powai, Mumbai 400076, Maharashtra, India.} 

\email{ronnie@math.iitb.ac.in} 	
	
\author[A. Tripathi]{Amit Tripathi} 
	
\address{Department of Mathematics, Indian Institute of Technology Hyderabad, Kandi, Sangareddy, 502285, Telangana, India.} 
	
\email{amittr@gmail.com}

\subjclass[2010]{14E20, 14J60, 14H50}
	
\keywords{Ulrich bundles, double planes, Cayley-Bacharach}

\begin{abstract}
We prove that a double cover of $\mathbb{P}^2$ ramified 
along a general smooth curve $B$ of degree $2s$, for $s\geq 3$, 
supports a rank $2$ special Ulrich bundle.
\end{abstract}	
	
\maketitle

\section{Introduction}
Throughout this paper we work over the field of complex numbers $\C$. 
Let $X$ be a $d$-dimensional smooth projective variety. 
Unless mentioned otherwise, $\mc O_X(1)$ will always 
denote an ample and globally generated line bundle on $X$.
\begin{definition}\label{def-Ulrich}
A locally free sheaf (vector bundle) $E$ 
on $X$ is said to be Ulrich with respect to $\mc O_X(1)$ 
(or simply Ulrich when the bundle $\mc O_X(1)$ is understood) if 
the following two conditions are satisfied
\begin{enumerate}
	\item $H^i(X,E(-i))=0$ for all $i>0$\,,
	\item $H^j(X,E(-j-1))=0$ for all $j<n$\,.
\end{enumerate}
\end{definition}
We refer the reader to \cite[\S 2]{AK}
for basic definitions. 
In the literature authors usually 
define Ulrich with respect to a very ample line bundle, 
kindly see the next section for some remarks related to this. 
A conjecture of Eisenbud and Schreyer \cite{ES} 
states that every smooth projective variety supports 
an Ulrich bundle. Several people have constructed Ulrich
bundles on particular varieties and we list a few. 
They have been shown to exist on complete 
intersections by \cite{HUB}, on curves and del Pezzo surfaces 
by \cite{ES}, on $K3$ surfaces by \cite{AFO} and \cite{Faenzi}, 
on abelian 
surfaces by \cite{Be-16}, on ruled surfaces by \cite{Ap-Co-Mi},
on regular surface by \cite{Cas-ns-1}, \cite{Cas-rs},
on surfaces with $p_g=0$ and $q=1$ by
\cite{Cas-ns-2}, on surfaces of maximal Albanese dimension 
and some irregular surfaces by 
\cite{Lopez-0}, \cite{Lopez}.
The above list is far from being complete and we refer the reader 
to the above papers, for example \cite{Cas-ns-2}, and the references therein for 
more results. In \cite[Theorem 4.3]{CoHu} the authors
show, over the field of complex numbers, the existence 
of Ulrich bundles of rank two in a sufficiently
ample embedding.

Recently Narayanan and Parameswaran \cite{PN} 
studied the existence of Ulrich line bundles on 
a double plane $\pi:X \rightarrow \mathbb{P}^2$
branched along a smooth curve $B \subset \mathbb{P}^2$
of degree $2s$. 
In \cite[Theorem 1.5]{PN}, they prove that for each 
$s\geq 3$, there are special 
classes of double planes which admit Ulrich line bundles.
In \cite[Theorem 1.4]{PN} they show that a double plane 
branched along a generic smooth curve of degree 
$2s$, where $s \geq 3$, does not support an Ulrich 
line bundle. 

Let $X$ be a surface and let $K_X$ be the canonical line bundle. 
An Ulrich bundle $E$ of rank 2, with respect to $\mc O_X(1)$,  
is called special Ulrich if it also satisfies 
${\rm det}(E)\cong K_X\otimes \mc O_X(3)$.  Let $\pi:X\to \P^2$ be a 
degree 2 cover which is branched along a smooth 
curve $B\subset \P^2$ of degree $2s$. For such a map 
denote $\mc O_X(1):=\pi^*\mc O_{\mb P^2}(1)$ and by an Ulrich 
(respectively, special Ulrich) bundle on $X$ we will always mean 
Ulrich (respectively, special Ulrich) with respect to 
$\mc O_X(1)$.
In this note we show the following.
\begin{theorem}\label{main-th-intro}
Let $\pi: X \rightarrow \mathbb{P}^2$ be a double 
cover branched along a generic smooth curve $B \subset \mathbb{P}^2$
of degree $2s$, where $s\geq 3$. 
Then $X$ admits a special rank 2 Ulrich bundle. 
\end{theorem}

To prove the above result we use two inputs. The first 
is the well known correspondence 
between zero dimensional subschemes satisfying the Cayley-Bacharach property
and global sections of a rank 2 vector bundle, see 
\cite[\S5]{Tan-V}. Let $F$ be the degree $2s$ homogeneous polynomial
which defines $B$. Using \cite{Tan-V} we 
first prove
\begin{theorem}[Theorem \ref{main-theorem}]
	Let $\pi:X\to \P^2$ be a degree 2 cover which is branched along a smooth 
	curve $B\subset \P^2$ of degree $2s$, where $s\geq 3$. 
	Let $F$ denote the polynomial of degree $2s$
	which defines $B$. Assume that there are two polynomials $F_1$ and $F_2$
	of degree $s$ such that $F\in (F_1,F_2)$. Then $X$ supports a special Ulrich bundle
	of rank 2. 
\end{theorem}
The second
input is the first point in \cite[Theorem 5.1]{Chiantini} 
which enables us to conclude that for the general degree $2s$ hypersurface 
$F$ we can find degree $s$ hypersurfaces $F_1$ and $F_2$  
such that $F\in (F_1,F_2)$. We do not know if this holds for all
smooth degree $2s$ hypersurfaces. 

It has been brought to our attention that Mohan Kumar, P. 
Narayanan and A.J. Parameswaran have proved, using a different 
method, that every double plane
cover supports a rank 2 Ulrich bundle. \\

\noindent
{\bf Acknowledgements}. We thank Enrico Carlini and Luca Chiantini
for several helpful discussions related to their article \cite{Chiantini}. 
We thank Gianfranco Casnati for several useful comments. 
We thank the referee for an extremely careful reading of this 
article and for numerous useful suggestions which have improved 
the exposition. 

\section{Existence of Ulrich bundles}
To show that a bundle $E$ on $X$ is Ulrich we will use the following 
criterion. 
\begin{lemma}
	Let $X$ be a $d$-dimensional smooth projective variety and 
	let $\pi:X\to \mb P^d$ be a surjective and finite map of degree $e$.
	A bundle $E$ on $X$ is Ulrich with respect to 
	$\pi^*\mc O_{\mb P^d}(1)$ if and only 
	if $\pi_*E\cong \mc O_{\mb P^d}^{e\cdot {\rm rank}(E)}$.
\end{lemma}
\begin{proof}
	Let us first assume that $\pi_*E\cong  \mc O_{\mb P^d}^{e\,{\rm rank}(E)}$.
	Since the map is finite, and using projection formula, 
	we have $H^i(X,E(k))=H^i(\mb P^d,\pi_*E(k))$ for 
	all $i,k\in \Z$. Since $\pi_*E\cong \mc O_{\mb P^d}^{e\,{\rm rank}(E)}$
	it is clear that the conditions in Definition \ref{def-Ulrich}
	are satisfied. 
	
	Conversely, assume that $E$ satisfies the conditions in Definition \ref{def-Ulrich}. 
	Then it follows that 
	$\pi_*E$ and $(\pi_*E)^\vee$ are 0-regular of rank $e\cdot{\rm rank}(E)$. 
	Thus, both of them are $m$-regular 
	for all $m\geq 0$. From this it easily follows that $H^i(\mb P^d,\pi_*E(k))=0$ 
	for all $k\in \Z$ and for all $1\leq i\leq d-1$, that is, $\pi_*E$ is an ACM 
	bundle. Now applying Horrocks criterion, see \cite{Hor} or
	\cite[Theorem 2.3.1]{Okonek} where it is stated more precisely, 
	we get that $\pi_*E$ is a direct sum of line bundles. 
	If $\mc O_{\mb P^d}(a)$ is a summand of $\pi_*E$ then we get
	that $H^i(\mb P^d,\mc O_{\mb P^d}(a-i))=0$ for all $i>0$ and 
	$H^j(\mb P^d,\mc O_{\mb P^d}(a-j-1))=0$ for all $j<d$. 
	In particular, by taking $i=d$ and $j=0$ it follows 
	that $a=0$. This shows that $\pi_*E\cong \mc O_{\mb P^d}^{e\cdot{\rm rank}(E)}$.
\end{proof}
As mentioned before, authors usually define Ulrich bundles with respect to very ample line bundles. 
However, the existence of an Ulrich bundle with respect to an ample 
and globally generated line bundle $L$ ensures that there are Ulrich 
bundles with respect to $L^{\otimes n}$ for all $n>0$, see \cite[Proposition 3]{AK}
and the remarks following it.\\

\subsection{\bf Double covers of $\P^2$}\label{double-planes-intro}. We briefly recall
the main properties of double covers which we will use. A general
reference for this is \cite[\S17]{Barth}.
Let $B\subset \P^2$ be a smooth curve of degree $2s$ defined by a 
homogeneous polynomial $F$. Let $\pi:X\to \P^2$
be the double cover of $\P^2$ branched along $B$, the construction of which 
we briefly explain for the benefit of the reader. 
Let $\mb A$ denote the total space of the line bundle $A=\mc O_{\mb P^2}(s)$,
$\pi:\mb A\to \mb P^2$ the projection and 
\begin{equation}\label{tautological-section}
	T\in H^0(\mb A, \pi^*A)
\end{equation} 
be the tautological section. Define $X$ to be the subvariety of 
$\mb A$ defined by the section 
$T^2 - \pi^*F\in H^0(\mb A, \pi^*A^{\otimes 2}) =H^0(\mb A, \pi^*\mc O_{\mb P^2}(2s))$. 
We will abuse notation and denote the composite $X\subset \mb A\to \mb P^2$ also by $\pi$. 
Then $\pi$ is a finite map of degree 2 between
projective varieties. We list the important 
properties of double covers that we will use, see \cite[Lemma 17.1, Lemma 17.2]{Barth}.
\begin{enumerate}
	\item If $B$ is smooth then $X$ is smooth, this is explained 
	in the sentence preceding \cite[Lemma 17.1]{Barth}.
	\item Let $R\subset X$ denote the reduced divisor $\pi^{-1}(B)$.
		Then $\pi^*\mc O_{\P^2}(s)\cong \mc O_X(R)$. 
	\item\label{canonical} The canonical bundle $K_X\cong \mc O_X(s-3)$.
	\item\label{splitting} $\pi_*\mc O_X=\mc O_{\P^2}\oplus \mc O_{\P^2}(-s)$.
\end{enumerate}

\subsection{\bf Existence of Ulrich bundles}.
In this section we shall prove that a general double plane cover supports a 
special Ulrich bundle of rank 2. We will use the result in \cite[Theorem 10]{Tan-V}
to construct a rank 2 bundle on $X$. For the benefit of the reader we state the main 
result from \cite{Tan-V} that we need, but first we 
recall some definitions from \cite[page 2]{Tan-V}.
Given a subscheme 
$Z_2 \subset Z_1$, the ``complement” $Z$ of $Z_2$ in $Z_1$ is the canonical
closed subscheme $Z \subset  Z_1$ with sheaf of ideals 
$I_{Z} = [I_{Z_1} : I_{Z_2}]$, that is, for any open
set $U\subset X$, we define
$$I_{Z}(U ) := \{g \in \mc O_X (U ) \,\vert\, gI_{Z_2}(U ) \subset I_{Z_1} (U )\}\,.$$
We call $Z$ the residual subscheme of $Z_2$ in $Z_1$ and denote it
by
$Z = Z_1 - Z_2$ in the statement of the next theorem, which 
is \cite[Theorem 10]{Tan-V}.
There are three equivalences, but we state only
two of these. 
\begin{theorem}[Theorem 10, \cite{Tan-V}]\label{th-TV}
	Let $X$ be a complex smooth projective variety of dimension $n\geq 2$. 
	Let $Z\subset X$ be a subscheme of pure codimension $2$. 
	Then the following are equivalent:
	\begin{enumerate}[(1)]
		\item $Z$ is the zero subscheme of a section of a rank 2 
			vector bundle $\mc E$.
		\item There are hypersurfaces $D_1,D_2,D_3$ such that 
		$D_1$ and $D_2$ have no common components, $Z=D_1\cap D_2-D_1\cap D_2\cap D_3$
		and such that $D_1\cap D_2\cap D_3$ is of pure codimension 2 and is Cohen-Macaulay.
	\end{enumerate}
Further, if (1) and (2) hold then ${\rm det}(\mc E)\equiv \mc O_X(D_1+D_2-D_3)$.
\end{theorem}
With notation as above let $\mc E$ be the bundle which sits in the following short 
exact syzygy sequence, this is explained just before  \cite[Theorem 10]{Tan-V}.
\begin{equation}\label{explicit-bundle}
0\to \mc E(-D_1-D_2) \to \bigoplus_i\mc O_{X}(-D_i)\to I_{D_1\cap D_2\cap D_3}\to 0\,.
\end{equation}

\begin{theorem}\label{main-theorem}
	Let $\pi:X\to \P^2$ be a degree 2 cover which is branched along a smooth 
	curve $B\subset \P^2$ of degree $2s$, where $s\geq 3$. 
	Let $F$ denote the polynomial of degree $2s$
	which defines $B$. Assume that there are two polynomials $F_1$ and $F_2$
	of degree $s$ such that $F\in (F_1,F_2)$. Then $X$ supports a special Ulrich bundle
	of rank 2. 
\end{theorem}
\begin{proof}
	First let us note that there is no non-constant polynomial $H$ 
	which divides both $F_1$ and $F_2$, or else $H$ will divide 
	$F$, contradicting the smoothness of $B$. Thus, 
	the subscheme of $\mb P^2$ defined by the ideal $(F_1,F_2)$ is 0-dimensional and is contained 
	in $B$. We denote this by $Z'$ the subscheme defined by the ideal $(F_1,F_2)$. 
	Consider the scheme theoretic inverse image 
	$Z_1:=\pi^{-1}(Z')$. 
	
	For $i=1,2$ take $H_i=\pi^*F_i\in H^0(X,\mc O_X(s))$. 
	If $D_i$ denotes the divisor defined by $H_i$ then $Z_1=D_1\cap D_2$
	in the notation of Theorem \ref{th-TV}.
	Take $H_3=T\in H^0(X,\mc O_X(s))$ (see equation \eqref{tautological-section},
	by abuse of notation we denote the restriction of $T$ to $X$ also by $T$) 
	and $Z_2$ to be the subscheme of $Z_1$ 
	defined by $H_3$, thus, $Z_2=D_1\cap D_2\cap D_3$, where $D_3$ is 
	the divisor defined by $H_3$. Let us compute the ideal 
	$I_Z:=[I_{Z_1}:I_{Z_2}]$.
	
	If $x,y,z$ are homogeneous coordinates on $\P^2$ then
	let $\C[x,y]$ (we abuse notation here and denote the affine coordinates also
	by $x$ and $y$) denote the coordinate 
	ring of the open set $\{z\neq 0\}$. 
	Denote by $f$ the equation $F$ in 
	$\C[x,y]$, similarly, for the other polynomials. 
	The inverse image of this open set in $X$ has coordinate ring 
	$$\C[x,y,t]/(t^2-f)\,.$$
	The ideal $I_{Z_1}=(f_1,f_2)$ and the ideal $I_{Z_2}=(f_1,f_2,t)$.
	If $I\subset J$ are two ideals in a ring then it is clear 
	that $I\subset [I:J]$. Thus, $I_{Z_1}\subset I_Z$.
	We claim that $t\in [I_{Z_1}:I_{Z_2}]$. An element of 
	$I_{Z_2}$ looks like $\lambda f_1+\mu f_2 + \theta t$. 
	Since $t^2=f\in (f_1,f_2)$ it follows that $t(\lambda f_1+\mu f_2 + \theta t)\in I_{Z_1}$
	that is, $t\in [I_{Z_1}:I_{Z_2}]$.
	Thus, $I_Z=(f_1,f_2,t)$.
	In particular, $Z=Z_2=D_1\cap D_2\cap D_3$.
	Thus, in the notation from \cite[\S1, page 2]{Tan-V},
	we may write 
	$$Z=D_1\cap D_2-D_1\cap D_2\cap D_3\,.$$
	Moreover, $D_1\cap D_2\cap D_3$ is of pure codimension 2 and is Cohen-Macaulay
	(since both depth and dimension are 0). 
	Thus, there is a bundle $\mc E$ which sits in the short exact sequence 
	\eqref{explicit-bundle} and
	has a global section 
	\begin{equation}\label{global-section}
		\mc O_X \to \mc E
	\end{equation}
	whose vanishing gives $Z$. 
	The line bundles $\mc O_X(D_i)$ 
	are all isomorphic to $\mc O_X(s)$.
	Thus, ${\rm det}(\mc E)=\mc O_X(s)$. 
	Twisting the short exact sequence \eqref{explicit-bundle} by $\mc O_X(2s)$
	we get a short exact sequence 
	\begin{equation}\label{explicit-bundle-1}
	0\to \mc E \to \mc O_X(s)^{\oplus 3} \to I_Z(2s)\to 0\,.
	\end{equation}
	Consider the commutative diagram
	\begin{equation*}
	\xymatrix{
		0 \ar[r]& I_{Z'} \ar[r]  \ar[d]^i& 
		\mathcal{O}_{\mathbb{P}^2} \ar[r] \ar[d]^{\pi^{\#}}& 
		\mathcal{O}_{Z'} \ar[d] \ar[r] & 0\\
		0 \ar[r]& \pi_*I_Z \ar[r]& 
		\pi_* \mathcal{O}_X \ar[r] & \pi_*\mathcal{O}_Z \ar[r] & 0 \,.
	}
	\end{equation*}
	In the notation we used above, the coordinate ring of 
	$Z'$ is $\C[x,y]/(f_1,f_2)$ and the coordinate ring of 
	$Z$ is $\C[x,y,t]/(f_1,f_2,t)\cong \C[x,y]/(f_1,f_2)$.
	Thus, the right vertical arrow is an isomorphism,
	which proves that the cokernel of $\pi^{\#}$ and $i$
	are isomorphic.

	Let $\sigma$ denote the involution of $X$ interchanging the two 
	sheets of the double cover.  
	One has the trace map ${\rm Tr}:\pi_*\mc O_X\to \mc O_{\P^2}$
	defined as follows. Let $U\subset \P^2$ be open and let 
	$f\in \mc O_X(\pi^{-1}(U))$. Define ${\rm Tr}(f):=f+f\circ \sigma\in \mc O_{\P^2}(U)$.
	Now if $f\in \mc O_{X}(\pi^{-1}(U))$ and $f$ vanishes 
	at a point $p\in Z$, then since $Z\subset B$ and $\sigma$ is 
	the identity on $B$, it follows that 
	$f\circ  \sigma$ also vanishes at $p$.
	From this it is clear that ${\rm Tr}$ maps $\pi_*I_Z$ to $I_{Z'}$.
	
	The map ${\rm Tr}$ gives a splitting of the map $\pi^{\#}$.
	Since {\rm Tr} maps $\pi_*I_Z$ to $I_{Z'}$ it follows that 
	${\rm Tr}$ gives a splitting of the map $i$.
	Thus, $\pi_*I_{Z}$ is isomorphic to the direct sum
	of $I_{Z'}$ and the cokernel of $i$.
	We saw above that the cokernel of $\pi^{\#}$
	and the cokernel of $i$ are isomorphic. 
	Since $\pi_*\mc O_X=\mc O_{\P^2}\oplus \mc O_{\P^2}(-s)$, see 
	property \eqref{splitting} in subsection \ref{double-planes-intro},
	it follows that the cokernel of $\pi^{\#}$ is isomorphic to 
	$\mc O_{\mb P^2}(-s)$. Thus, 
	it follows that 
	\begin{equation}\label{main-ses}
	\pi_*I_Z\cong I_{Z'}\oplus \mc O_{\mb P^2}(-s)\,.
	\end{equation}
	Thus, we have 
	\begin{equation}\label{h^0(I_Z(2s))}
		h^0(X,I_Z(2s))=h^0(\P^2,I_{Z'}(2s))+h^0(\P^2,\mc O_{\P^2}(s))\,.
	\end{equation}
	Using the short exact sequence 
	\begin{equation}\label{exact-seq-I_Z'}
	0\to \mc O_{\P^2}(-2s)\to \mc O_{\P^2}(-s)^{\oplus 2}\to I_{Z'}\to 0\,,
	\end{equation}
	we get that $h^0(\P^2,I_{Z'}(2s))=2h^0(\P^2,\mc O_{\P^2}(s))-1$.
	Using equation \eqref{h^0(I_Z(2s))} we get 
	$$h^0(X,I_Z(2s))=3h^0(\P^2,\mc O_{\P^2}(s))-1\,.$$
	Next we will compute $H^0(X,\mc E)$. Taking dual of \eqref{global-section}
	we get an exact sequence 
	$$0\to {\rm det}(\mc E)^\vee\to \mc E^\vee\to I_Z\to 0\,.$$
	Since ${\rm det}(\mc E)=\mc O_X(s)$ and $\mc E$ is of rank 2, we get 
	$\mc E^\vee = \mc E \otimes {\rm det}(\mc E)^\vee$, which gives 
	$$0 \to \mc O_X \to \mc E \to I_Z(s) \to 0\,.$$
	Applying $\pi_*$ to this we get 
	\begin{equation}\label{Koszul-equation}
	0\to \pi_*\mc O_X \to \pi_* \mc E \to \pi_*I_Z(s)\to 0\,.
	\end{equation}
	Using \eqref{exact-seq-I_Z'} we get $h^0(\P^2,I_{Z'}(s))=2$  
	and using \eqref{main-ses}  
	we get that 
	$$h^0(\P^2, \pi_*I_Z(s))=3\,.$$ From this it follows that 
	$h^0(\P^2,\pi_*\mc E)=4$. Applying $\pi_*$ to \eqref{explicit-bundle-1}
	and taking cohomology we get
	\begin{align*}
		h^1(\P^2,\pi_*\mc E)&=h^0(\P^2,\pi_*I_Z(2s))+h^0(\P^2,\pi_*\mc E) - 3-3h^0(\P^2,\mc O_{\P^2}(s))\\
			&=3h^0(\P^2,\mc O_{\P^2}(s))-1 + 4- 3-3h^0(\P^2,\mc O_{\P^2}(s))\\
			&=0\,.
	\end{align*}
	Consider the commutative diagram
	\[
	\xymatrix{
		H^0(\P^2,\pi_*\mc E)\ar[r]\ar[dr] & H^0(\P^2,\pi_*I_Z(s))\ar[d]\\
			&	H^0(\P^2,\mc O_{\P^2})
	}
	\]
	The vertical arrow is the projection in equation \eqref{main-ses}
	and is surjective. The horizontal arrow is surjective as is easily 
	seen by taking cohomology of the sequence 
	\eqref{Koszul-equation}. Thus, we have a map $\pi_*\mc E\to \mc O_{\P^2}$ 
	which induces a surjection on global sections. This shows that this map 
	is split.
	Thus, 
	$$\pi_*\mc E=\mc G \oplus \mc O_{\P^2}\,,$$
	where $\mc G$ is a locally free sheaf and sits in a short exact sequence 
	$$0\to \pi_*\mc O_X \to \mc G \to I_{Z'}(s)\to 0\,.$$
	We will now show that $\mc G$ is trivial. Consider the following pullback diagram.
	\begin{equation*} \label{dgm_pullback}
	\xymatrix{ 
		0 \ar[r]& \pi_* \mathcal{O}_X \ar[r] \ar@{=}[d]& 
		\mathcal{F} \ar[r] \ar[d]^a& 
		\mc O_{\P^2}^{\oplus 2} \ar[r] \ar[d]^b & 0 \\ 
		0 \ar[r] & \pi_* \mathcal{O}_X \ar[r] & 
		\mc G \ar[r] & 
		I_{Z'}(s) \ar[r]& 0 
	}
	\end{equation*} 
	From this it follows that $\mc F=\mc O_{\P^2}^{\oplus 3}\oplus \mc O_{\P^2}(-s)$
	since ${\rm Ext}^1(\mc O_{\P^2},\pi_*\mc O_X)=0$.
	We may split the top row and compose the splitting with $a$ to get a diagram 
	\begin{equation*} 
	\xymatrix{ 
		0 \ar[r]& \mc O_{\P^2}(-s) \ar[r] \ar[d]^d& \mc O_{\P^2}^{\oplus 2} 
		\ar[r] \ar[d]^c & I_{Z'}(s) \ar[r]\ar@{=}[d] & 0 \\ 
		0 \ar[r] & \pi_* \mathcal{O}_X \ar[r] & 
		\mc G \ar[r] & 
		I_{Z'}(s) \ar[r]& 0 
	}
	\end{equation*} 
	Suppose ${\rm Ker}\,c\neq 0$, then the image of $c$ is a sheaf of rank 1,
	which surjects onto $I_{Z'}(s)$. This forces that the image is isomorphic to 
	$I_{Z'}(s)$, which defines a splitting of the bottom row. However, since 
	$\mc G$ is locally free, this is not possible. Thus, ${\rm Ker}\,c=0$.
	
	Now let us consider the left vertical arrow 
	$d:\mc O_{\P^2}(-s)\to \mc O_{\P^2}\oplus \mc O_{\P^2}(-s)$. If the cokernel
	is $\mc O_{\P^2}$ then we get that $\mc G$ is the trivial bundle. The only 
	other possbility for the cokernel is $\mc O_C\oplus \mc O_{\P^2}(-s)$, where
	$C$ is a hypersurface of degree $s$ in $\P^2$. In this case, $\mc G$ sits in a 
	sequence 
	$$0\to \mc O_{\P^2}^{\oplus 2}\to \mc G \to \mc O_C\oplus \mc O_{\P^2}(-s)\to 0\,.$$
	Since $\mc G$ is a summand of $\pi_*\mc E$ and $H^1(\P^2,\pi_*\mc E)=0$,
	it follows that $H^1(\P^2,\mc G)=0$. This forces that 
	$$H^1(\P^2,\mc O_C)=0\,.$$
	But now using $0\to \mc O_{\P^2}(-s)\to \mc O_{\P^2}\to \mc O_C\to 0$ 
	we get $0=H^1(\P^2,\mc O_C)=H^2(\P^2,\mc O_{\P^2}(-s))=H^0(\P^2,\mc O_{\P^2}(s-3))^\vee$,
	which is not possible if $s\geq 3$. Thus, the cokernel of $d$ 
	is $\mc O_{\P^2}$ and so $\mc G$ and $\pi_*\mc E$ are trivial. 
	This proves that $\mc E$ is an Ulrich bundle on $X$. 
	
	Recall from \eqref{canonical} that the canonical line bundle of $X$ is $\mc O_X(s-3)$.
	Since ${\rm det}(\mc E)\cong \mc O_X(s)=\mc O_X(s-3)\otimes \mc O_X(3)$, it follows 
	that $\mc E$ is a special Ulrich bundle. 
\end{proof}

\begin{proof}[Proof of Theorem \ref{main-th-intro}]
For the convenience of the reader we recall the statement from 
\cite{Chiantini} that we are using. Their main result 
gives a description of all the possible complete intersections
of codimension $r$ that can be found on a general hypersurface
of degree $d$ in $\mb P^n$ when $2r\leq n+2$. In our situation $n=r=2$
and $d=2s$. In this situation 
the first point in \cite[Theorem 5.1]{Chiantini} 
enables us to conclude that for the general degree $2s$ hypersurface 
$F$ we can find degree $s$ hypersurfaces $F_1$ and $F_2$  
such that $F\in (F_1,F_2)$. Now we apply Theorem \ref{main-theorem}
and get that $X$ supports a special Ulrich bundle of rank 2.
This proves Theorem \ref{main-th-intro}.
\end{proof}

\end{document}